\documentclass[reqno]{amsart}
\usepackage{amsmath,amsthm,amssymb,bbm,color,verbatim,tikz}
\usepackage{amsaddr}
\input xy 
\xyoption{all}

\numberwithin{equation}{section}

 \renewcommand{\marginpar}[2][]{}

\newcommand{\CARD}{{\rm CARD}}
\newcommand{\REG}{{\rm REG}}

\newcommand{\GCH}{{\rm GCH}}

\newcommand{\ZFC}{{\rm ZFC}}

\renewcommand{\P}{{\mathbb P}}
\newcommand{\Q}{{\mathbb Q}}

\newcommand{\F}{{\mathbb F}}

\newcommand{\Add}{\mathop{\rm Add}}

\newcommand{\forces}{\Vdash}

\newcommand{\restrict}{\upharpoonright}

\newcommand{\<}{\langle}
\renewcommand{\>}{\rangle}

\newcommand{\elesub}{\prec}

\newcommand{\dom}{\mathop{\rm dom}}
\newcommand{\ran}{\mathop{\rm ran}}
\newcommand{\tail}{\text{tail}}

\newcommand{\cf}{\mathop{\rm cf}}

\newcommand{\crit}{\mathop{\rm crit}}

\renewcommand{\and}{\mathop{\&}}

\newcommand{\Mos}{\pi}
\newcommand{\invMos}{k}

\newtheorem{theorem}{Theorem}
\newtheorem{lemma}{Lemma}
\newtheorem{corollary}[theorem]{Corollary}
\newtheorem{claim}{Claim}

\newtheorem*{theorem1}{Theorem 1}

\theoremstyle{definition}
\newtheorem{question}{Question}

\date{\today}

\begin{document}

\title{On supercompactness and the continuum function}

\author[Brent Cody]{Brent Cody$^1$}
\thanks{The first author would like to thank the Fields Institute as well as the University of Prince Edward Island for their support while this work was being carried out.}
\address[Brent Cody]{ 
The Fields Institute for Research in Mathematical Sciences\\
222 College Street\\
Toronto, Ontario M5S 2N2, Canada\\
+1-647-983-1766\\
{\tt brentcody@gmail.com}\\
{\tt http://www.fields.utoronto.ca/$\sim$bcody/}
} 
\thanks{$^1$Current address: The University of Prince Edward Island, 550 University Avenue, Charlottetown
Prince Edward Island, Canada C1A 4P3}

\author[Menachem Magidor]{Menachem Magidor}
\address[Menachem Magidor]{ 
Einstein Institute of Mathematics\\
The Hebrew University of Jerusalem\\
Jerusalem, 91904, Israel\\
+972-2-65-84143\\
{\tt mensara@savion.huji.ac.il}\\
} 

\begin{abstract}

Given a cardinal $\kappa$ that is $\lambda$-supercompact for some regular cardinal $\lambda\geq\kappa$ and assuming $\GCH$, we show that one can force the continuum function to agree with any function $F:[\kappa,\lambda]\cap\REG\to\CARD$ satisfying $\forall\alpha,\beta\in\dom(F)$ $\alpha<\cf(F(\alpha))$ and $\alpha<\beta$ $\implies$ $F(\alpha)\leq F(\beta)$, while preserving the $\lambda$-supercompactness of $\kappa$ from a hypothesis that is of the weakest possible consistency strength, namely, from the hypothesis that there is an elementary embedding $j:V\to M$ with critical point $\kappa$ such that $M^\lambda\subseteq M$ and $j(\kappa)>F(\lambda)$. Our argument extends Woodin's technique of surgically modifying a generic filter to a new case: Woodin's key lemma applies when modifications are done on the range of $j$, whereas our argument uses a new key lemma to handle modifications done off of the range of $j$ on the ghost coordinates. This work answers a question of Friedman and Honzik \cite{FriedmanHonzik:SupercompactnessAndFailuresOfGCH}. We also discuss several related open questions.


\end{abstract}

\maketitle

\section{Introduction}\label{sectionintroduction}

The behavior of the continuum function $\gamma\mapsto 2^\gamma$ on the regular cardinals was shown, by Easton, to be highly undetermined by the axioms of $\ZFC$. Easton proved \cite{Easton:PowersOfRegularCardinals} that if $F$ is any function from the regular cardinals to the cardinals satisfying $\alpha<\cf(F(\alpha))$ and $\alpha<\beta$ $\implies$ $F(\alpha)\leq F(\beta)$, then there is a cofinality-preserving forcing extension in which $2^\gamma=F(\gamma)$ for every regular cardinal $\gamma$. Large cardinal axioms impose additional restrictions on the continuum function on the regular cardinals. For example, if $\kappa$ is a supercompact cardinal and $\GCH$ holds below $\kappa$, then $\GCH$ holds everywhere. It therefore seems natural to address the question: What functions can be forced to coincide with the continuum function on the regular cardinals while preserving large cardinals? From what hypotheses? In particular, let us consider the following question.
\begin{question} \label{questionsceaston}
Given a $\lambda$-supercompact cardinal $\kappa$ where $\lambda\geq\kappa$ is a cardinal, and assuming $\GCH$, what functions $F$ from the regular cardinals to the cardinals can be forced to equal the continuum function on the interval $[\kappa,\lambda]$ while preserving the $\lambda$-supercompactness of $\kappa$ and preserving cardinals? From what hypotheses? 

\end{question}

Menas \cite{Menas:ConsistencyResultsConcerningSupercompactness} proved that, assuming $\GCH$, one can force the continuum function to agree at every regular cardinal with any \emph{locally definable}\footnote{A function $F$ is \emph{locally definable} if there is a sentence $\psi$, true in $V$, and a formula $\varphi(x,y)$ such that for all cardinals $\gamma$, if $H_\gamma\models \psi$, then $F$ has a closure point at $\gamma$ and for all $\alpha,\beta<\gamma$, we have $F(\alpha)=\beta\leftrightarrow H_\gamma\models \varphi(\alpha,\beta)$.} function $F$ satisfying the requirements of Easton's theorem, while preserving all cofinalities and preserving the supercompactness of a cardinal $\kappa$. If $\kappa$ is a Laver-indestructible supercompact cardinal in a model $V$, as in \cite{Laver:MakingSupercompactnessIndestructible}, it easily follows that one can force over this model to achieve any reasonable behavior of the continuum function at and above $\kappa$, while preserving the supercompactness of $\kappa$. In particular, starting with a Laver-indestructible supercompact cardinal, one can obtain a model with a measurable cardinal at which $\GCH$ fails. However, one can also obtain a model with a measurable cardinal at which $\GCH$ fails from a much weaker large cardinal assumption. Woodin proved that the existence of a measurable cardinal at which $\GCH$ fails is equiconsistent with the existence of an elementary embedding $j:V\to M$ with critical point $\kappa$ such that $M^\kappa\subseteq M$ and $j(\kappa)>\kappa^{++}$ (see \cite[Theorem 25.1]{Cummings:Handbook}). Woodin's argument illustrates that under certain conditions, one may perform a type of surgical modification on a generic filter $g$ to obtain $g^*$ in order to meet the lifting criterion, $j"G\subseteq g^*$, and such that $g^*$ remains generic. In Woodin's proof, the modifications made to $g$ in order to obtain $g^*$ only occur on the range of $j$, and his key lemma shows that such changes are relatively mild in the sense that for a given condition $p\in g$, the set over which modifications are made to obtain $p^*\in g^*$ has size at most $\kappa$. Hamkins showed \cite{Hamkins:TallCardinals} that Woodin's method could be applied to obtain an indestructibility theorem for tall cardinals.

The first author extended Woodin's surgery method to the case of partially supercompact cardinals in \cite{Cody:TheFailureOfGCHAtADegreeOfSupercompactness}. It is shown in \cite{Cody:TheFailureOfGCHAtADegreeOfSupercompactness} that the existence of a $\lambda$-supercompact cardinal $\kappa$ such that $2^\lambda\geq\lambda^{++}$ is equiconsistent with the following hypothesis.
\begin{quote}$(*)$
There is an elementary embedding $j:V\to M$ with critical point $\kappa$ such that $M^\lambda\subseteq M$ and $j(\kappa)>\lambda^{++}$. 
\end{quote} 
The method used in \cite{Cody:TheFailureOfGCHAtADegreeOfSupercompactness} is to, after a suitable preparatory iteration, blow up the size of the powerset of $\kappa$ using Cohen forcing, in order to achieve $2^\kappa=\lambda^{++}$ and then use Woodin's method of surgery to lift the elementary embedding. Thus, one obtains a model in which $\kappa$ is $\lambda$-supercompact and $\GCH$ fails at $\lambda$, \emph{because} $2^\kappa=\lambda^{++}$. Answering a question posed in \cite{Cody:TheFailureOfGCHAtADegreeOfSupercompactness}, Friedman and Honzik \cite{FriedmanHonzik:SupercompactnessAndFailuresOfGCH} used a variant of Sacks forcing for uncountable cardinals to show, from the hypothesis $(*)$, one can obtain a forcing extension in which $\kappa$ is $\lambda$-supercompact, $\GCH$ holds on $[\kappa,\lambda)$, and $2^\lambda\geq\lambda^{++}$. The methods of both \cite{Cody:TheFailureOfGCHAtADegreeOfSupercompactness} and \cite{FriedmanHonzik:SupercompactnessAndFailuresOfGCH} leave open the following question, which appears in \cite{FriedmanHonzik:SupercompactnessAndFailuresOfGCH}. Assuming $\GCH$ and $(*)$, where $\kappa<\gamma<\lambda$ are regular cardinals, is there a cofinality-preserving forcing extension in which $\kappa$ remains $\lambda$-supercompact, $\GCH$ holds on the interval $[\kappa,\gamma)$, and $2^\gamma=\lambda^{++}$? In this article we answer this question, and indeed, provide a full answer to Question \ref{questionsceaston}, by proving the following theorem.

\begin{theorem} \label{maintheorem}
Suppose $\GCH$ holds, $\kappa<\lambda$ are regular cardinals, and $F:[\kappa,\lambda]\cap\REG\to\CARD$ is a function such that $\forall\alpha,\beta\in\dom(F)$ $\alpha<\cf(F(\alpha))$ and $\alpha<\beta$ $\implies$ $F(\alpha)\leq F(\beta)$. If there is an elementary embedding $j:V\to M$ with critical point $\kappa$ such that $M^\lambda\subseteq M$ and $j(\kappa)>F(\lambda)$, then there is a cofinality-preserving forcing extension in which $\kappa$ remains $\lambda$-supercompact and $2^\gamma=F(\gamma)$ for every regular cardinal $\gamma\in [\kappa,\lambda]$.
\end{theorem}

The stage $\kappa$ forcing $\Q$ in our proof will be an Easton-support product forcing for adding $F(\gamma)$-subsets to $\gamma$ for each regular cardinal $\gamma\in[\kappa,\lambda]$. Since conditions $p$ in this forcing will have size greater than the critical point of the embedding $j$ under consideration, it follows that $j(p)\neq j"p$ and hence $j(p)$ will have ghost-coordinates. Part of the modification we will perform on a certain generic filter $g$ to obtain $g^*$ with $j"H\subseteq g^*$ will occur off of the range of $j$ on these ghost coordinates. Hence, Woodin's key lemma does not apply. We will formulate and prove a new key lemma (Lemma \ref{newkeylemmaeaston} below) that allows much of the rest of the argument to be carried out as before. Essentially, our new key lemma shows that the set over which modifications are made to a condition $p\in g$ to obtain $p^*\in g^*$ can be broken up into $\lambda$ pieces, each of which is in the relevant model.

The hypothesis in Theorem \ref{maintheorem} is of optimal consistency strength for the simple reason that if $j:V\to M$ witnesses that $\kappa$ is $\lambda$-supercompact, then
$$2^\lambda\leq (2^\lambda)^M < j(\kappa).$$

In Section \ref{sectionopenquestions} we discuss an easy corollary of Theorem \ref{maintheorem} and an open question both addressing the case when $\lambda$ is a singular cardinal.

\section{Preliminaries}\label{sectionpreliminaries}

\subsection{Background material}

We assume familiarity with the large cardinal notions of measurable and supercompact cardinals as well as with the characterization of these notions in terms of the existence of nontrivial elementary embeddings from the universe $V$ into transitive inner models $M\subseteq V$. Some familiarity with the techniques of lifting large cardinal embeddings to forcing extensions  is also assumed (see \cite[Sections 8 and 9]{Cummings:Handbook} and \cite[Chapter 21]{Jech:Book}).

\subsection{Supercompactness with tallness}

We say that a cardinal $\kappa$ is \emph{$\theta$-tall} if $\theta>\kappa$ is an ordinal and there is an elementary embedding $j:V\to M$ with critical point $\kappa$ such that $j(\kappa)>\theta$ and $M^\kappa\subseteq M$. We say that $\kappa$ is \emph{$\lambda$-super\-compact with tallness $\theta$} if $\kappa\leq\lambda$ are cardinals, $\theta>\lambda$ is an ordinal, and there is a $j:V\to M$ with critical point $\kappa$ such that $M^\lambda\subseteq M$ and $j(\kappa)>\theta$. Such cardinals have been studied by Hamkins \cite{Hamkins:TallCardinals}, Cody \cite{Cody:TheFailureOfGCHAtADegreeOfSupercompactness}, Friedman-Honzik \cite{FriedmanHonzik:SupercompactnessAndFailuresOfGCH}, and others.

The next lemma appears in \cite{Cody:TheFailureOfGCHAtADegreeOfSupercompactness}, and is easy to verify by factoring an elementary embedding through the ultrapower by an extender.
\begin{lemma} \label{lemmaextender1}
If $\kappa$ is $\lambda$-supercompact with tallness $\theta$ then there is an embedding $j:V\to M$ witnessing this such that 
$$M=\{j(h)(j"\lambda,\alpha)\mid h:P_\kappa\lambda\times\kappa\to V\ \and\ \alpha<\theta^\lambda\ \and\  h\in V\}.$$
\end{lemma}

\subsection{Easton's product forcing}
Let us quickly review Easton's product forcing and fix some notation. Suppose $\kappa$ and $\lambda$ are regular cardinals and $F:[\kappa,\lambda]\cap\REG\to\CARD$ is a function satisfying Easton's requirements; that is, for all regular cardinals $\alpha,\beta\in[\kappa,\lambda]$ one has
\begin{align}
\alpha &<\cf(F(\alpha)) \textrm{ and} \tag{E1}\\
\alpha &<\beta\implies F(\alpha)\leq F(\beta).\tag{E2}
\end{align}
Let $\Q^F_{[\kappa,\lambda]}$ denote the Easton support product of Cohen forcing to achieve $2^\gamma=F(\gamma)$ for every regular $\gamma\in [\kappa,\lambda]$, by forcing over a model of $\GCH$. We can regard conditions $p\in \Q^{F}_{[\kappa,\lambda]}$ as functions satisfying the following.

\begin{itemize}
\item Every element in $\dom(p)$ is of the form $(\gamma,\alpha,\beta)$ where $\gamma\in [\kappa,\lambda]$ is a regular cardinal, $\alpha<\gamma$, and $\beta<F(\gamma)$.
\item (Easton support) For each regular cardinal $\gamma\in[\kappa,\lambda]$ we have 
$$|\{(\delta,\alpha,\beta)\in\dom(p)\mid\delta\leq\gamma\}|<\gamma.$$
\item $\ran(p)\subseteq\{0,1\}$.
\end{itemize}

Let $\dom(\Q^F_{[\kappa,\lambda]}):=\bigcup\{\dom(p)\mid p\in \Q^F_{[\kappa,\lambda]}\}$ and notice that 
$$\dom(\Q^F_{[\kappa,\lambda]})=\{(\delta,\alpha,\beta)\mid\textrm{$\delta\in[\kappa,\lambda]\cap\REG$, $\alpha<\delta$, and $\beta<F(\delta)$}\}$$
By \cite{Easton:PowersOfRegularCardinals} we obtain the following.
\begin{lemma}\label{lemmaeaston}
Assuming $\GCH$, forcing with the poset $\Q^F_{[\kappa,\lambda]}$ preserves all cofinalities and achieves $2^\gamma=F(\gamma)$ for every regular cardinal $\gamma\in[\kappa,\lambda]$ while preserving $\GCH$ otherwise.
\end{lemma}
One may also see \cite[Theorem 15.18]{Jech:Book} for further details regarding the partial order $\Q^F_{[\kappa,\lambda]}$ and Lemma \ref{lemmaeaston}.

\section{Proof of the main theorem}

The next lemma will allow us to define the forcing iteration we will use to prove Theorem \ref{maintheorem}. Notice that if $\GCH$ holds and $\kappa$, $\lambda$, $F$, and $j$ are as in the statement of Theorem \ref{maintheorem} then $F(\lambda)^\lambda=F(\lambda)$ since $\lambda<\cf(F(\lambda))$. Thus, under the hypothesis of Theorem \ref{maintheorem}, it follows from Lemma \ref{lemmaextender1} that we may assume without loss of generality that $j:V\to M$ is such that 
\begin{align}
M&=\{j(h)(j"\lambda,\alpha)\mid h:P_\kappa\lambda\times\kappa\to V\ \and\ \alpha<F(\lambda)\ \and\  h\in V\} \label{eqnextender}
\end{align}

\begin{lemma}\label{lemmafffeaston}
Assume $\GCH$ and that $\kappa,\lambda,F$, and $j:V\to M$ are as in the hypothesis of Theorem \ref{maintheorem}. Furthermore, assume that $M$ is as in (\ref{eqnextender}). Then there is a cofinality-preserving forcing extension $V[G]$ such that the following hold.
\begin{enumerate}
\item $\GCH$ holds in $V[G]$.
\item There is a partial function $f$ from $\kappa$ to $V_\kappa$ in $V[G]$ such that $j$ lifts to $j:V[G]\to M[j(G)]$ and $j(f)(\kappa)=F$.
\item $M[j(G)]=\{j(h)(j"\lambda,\alpha)\mid h:(P_\kappa\lambda)^V\times\kappa\to V\ \and\ \alpha<F(\lambda)\ \and\  h\in V[G]\}$ \label{extenderpreserved}
\end{enumerate}
\end{lemma}

\begin{proof}

Fix $\vec{x}=\<x_\xi\mid\xi<\kappa\>$, a well-ordering of $V_\kappa$. Then $j(\vec{x})=\<z_\xi\mid\xi<j(\kappa)\>$ is a well ordering of $V_{j(\kappa)}^M$. Since $M^\lambda\subseteq M$ it is clear that $F\in M$ and hence $F\in V_{j(\kappa)}^M$. Fix $\alpha<j(\kappa)$ such that $F=z_\alpha$. Let $\F_\kappa$ be Woodin's poset for adding a partial function from $\kappa$ to $\kappa$, defined as follows. Conditions in $\F_\kappa$ are partial functions $p\subseteq \kappa\times\kappa$ satisfying the following conditions.
\begin{itemize}
\item $\gamma\in\dom(p)$ $\implies$ $\gamma<\kappa$ is an inaccessible cardinal and $p"\gamma\subseteq\gamma$.
\item For every inaccessible cardinal $\mu<\kappa$ one has $|p\restrict\mu|<\mu$.
\end{itemize}
The ordering on $\F_\kappa$ is defined by $p\leq p'$ if and only if $p\supseteq p'$. For a proof that under $\GCH$ the poset $\F_\kappa$ preserves all cofinalities and does not disturb the continuum function, see \cite[Theorem 1.3]{Hamkins:TheLotteryPreparation}. Let $G$ be $V$-generic for $\F_\kappa$ and let $\tilde{f}=\bigcup G$. Clearly $\tilde{f}$ is a partial function from $\kappa$ to $\kappa$. 

Let us show that in $M$ the poset $j(\F_\kappa)$ factors below a condition $p_0$ as $j(\F_\kappa)/p_0\cong \F_\kappa\times \F_{\tail}$ such that $\F_{\tail}$ is ${\leq}F(\lambda)$-closed in $M$. Notice that $\{(\kappa,\alpha)\}$ is a condition in $\F_\kappa$ and that the forcing $j(\F_\kappa)$ can be factored below $\{(\kappa,\alpha)\}$ as $j(\F_\kappa)/\{(\kappa,\alpha)\}\cong \F_\kappa \times \F_{[\gamma_0,j(\kappa))}$ where $\gamma_0$ denotes the least $M$-inaccessible cardinal greater than $\alpha$, and $\F_{[\gamma_0,j(\kappa))}$ is a ${\leq}\gamma_0$-closed poset in $M$. If $\gamma_0\geq F(\lambda)$ then we can take $p_0=\{(\kappa,\alpha)\}$ and $\F_{\tail}=\F_{[\gamma_0,j(\kappa))}$. Assuming $\gamma_0<F(\lambda)$, let $p_0=\{(\kappa,\alpha),(\gamma_0,F(\lambda))\}$ and factor $j(\F_\kappa)$ below the condition $p_0=\{(\kappa,\alpha),(\gamma_0,F(\lambda))\}$ to obtain $$j(\F_\kappa)/\{(\kappa,\alpha),(\gamma_0,F(\lambda))\}\cong\F_\kappa\times \F_{[\gamma_1,j(\kappa))}$$
where $\gamma_1$ is the least $M$-inaccessible cardinal greater than $F(\lambda)$ and $\F_{[\gamma_1,j(\kappa))}$ is ${\leq}\gamma_1$-closed in $M$. In this case we take $\F_{\tail}=\F_{[\gamma_1,j(\kappa))}$.

Now we will show that there is an $M$-generic filter $K$ for $j(\F_\kappa)/p_0\cong\F_\kappa\times\F_{\tail}$ in $V[G]$ satisfying the lifting criterion $j"G\subseteq K$, and hence $j$ lifts to $j:V[G]\to M[j(G)]$ (see \cite[Proposition 9.1]{Cummings:Handbook}) where $j(G)$ is the upward closure of $K$ in $j(\F_\kappa)$. This will suffice for (2) because $p_0\in j(G)$ ensures that $j(\tilde{f})(\kappa)=\alpha$, and then working in $V[G]$, one can define a function $f$ with $\dom(f)\subseteq\kappa$ by $f(\xi)=x_{\tilde f(\xi)}$, and this function satisfies $j(f)(\kappa)=z_{j(\tilde f)(\kappa)}=z_\alpha=F$.

Let us show that $K$, as above, can be built in $V[G]$. First we show that there is an $M$-generic filter $G_{\tail}$ for $\F_{\tail}$ in $V$. Let $X$ denote the set of all dense subsets of all tails of the forcing $\F_\kappa$. It follows that $X$ has size at most $2^\kappa=\kappa^+$ in $V$. Since $M$ is as in (\ref{eqnextender}), every dense subset $D$ of $\F_{\tail}$ in $M$ is represented by a function $h:P_\kappa\lambda\times\kappa\to X$ as $D=j(h)(j"\lambda,\alpha)$ where $\alpha<F(\lambda)$ and $h\in V$. Since there are at most $(2^\kappa)^{\lambda^{<\kappa}}=(2^{\kappa})^\lambda=2^\lambda=\lambda^+$ such functions we can enumerate them as $\<h_\xi\mid\xi<\lambda^+\>\in V$. Working in $V$ we will build a decreasing sequence $\<p_\xi\mid\xi<\lambda^+\>$ with $p_\xi\in \F_{\tail}$ such that $p_\xi$ meets every dense subset of $\F_{\tail}$ in $M$ represented by some $h_\zeta$ for $\zeta\leq\xi$.  Assume $p_\zeta$ has been constructed for $\zeta<\xi$. At stage $\xi<\lambda$, since $M^\lambda\subseteq M$ in $V$, we may let $r_\xi\in \F_{\tail}$ be a lower bound of all previously constructed conditions. Since $\<j(h_\xi)(j"\lambda,\alpha)\mid\alpha<F(\lambda)\>\in M$ and $\F_{\tail}$ is ${\leq}F(\lambda)$-closed in $M$, it follows that there is a single condition $p_\xi$ below $r_\xi$ meeting every dense subset of $\F_{\tail}$ in the sequence $\<j(h_\xi)(j"\lambda,\alpha)\mid\alpha<F(\lambda)\>$. Since every dense subset of $\F_{\tail}$ in $M$ has a name that is represented by some function on our list $\<h_\xi\mid\xi<\lambda^+\>$, we can use $\<p_\xi\mid\xi<\lambda^+\>$ to generate $G_{\tail}\in V$ an $M$-generic filter for $\F_{\tail}$. Now let $K=G\times G_{\tail}$. It is easy to verify that $K$ is $M$-generic for $j(\F_{\kappa})/p_0\cong\F_\kappa\times F_{\tail}$.

We now show that (\ref{extenderpreserved}) holds. If $x\in M[j(G)]$ then $x$ has a $j(\F_\kappa)$-name, $\dot{x}\in M$ such that $x=\dot{x}_{j(G)}$. Thus we may write $\dot{x}=j(h)(j"\lambda,\alpha)$ for some $h:P_\kappa\lambda\times\kappa\to V$, $\alpha<F(\lambda)$, and $h\in V$. Working in $V[G]$ define a function $\tilde{h}:(P_\kappa\lambda)^V\times\kappa\to V[G]$ by letting $\tilde{h}(z,\xi)=h(z,\xi)_G$ whenever $h(z,\xi)$ is an $\F_\kappa$-name and letting $\tilde{h}(z,\xi)=\emptyset$ otherwise. Then $\tilde{h}\in V[G]$ and $j(\tilde{h})(j"\lambda,\alpha)=j(h)(j"\lambda,\alpha)_{j(G)}=\dot{x}_{j(G)}=x$. \end{proof}

Let us restate the main theorem.

\begin{theorem1} 
Suppose $\GCH$ holds, $\kappa<\lambda$ are regular cardinals, and $F:[\kappa,\lambda]\cap\REG\to\CARD$ is a function such that $\forall\alpha,\beta\in\dom(F)$ $\alpha<\cf(F(\alpha))$ and $\alpha<\beta$ $\implies$ $F(\alpha)\leq F(\beta)$. If there is an elementary embedding $j:V\to M$ with critical point $\kappa$ such that $M^\lambda\subseteq M$ and $j(\kappa)>F(\lambda)$, then there is a cofinality-preserving forcing extension in which $\kappa$ remains $\lambda$-supercompact and $2^\gamma=F(\gamma)$ for every regular cardinal $\gamma\in [\kappa,\lambda]$.
\end{theorem1}

\begin{proof} To simplify notation later on, let $\widetilde{V}$ denote the model we start with and let $\widetilde{j}:\widetilde{V}\to\widetilde{M}$ be an elementary embedding, which is a definable class of $\widetilde{V}$ such that (1) $\widetilde{M}^\lambda\subseteq \widetilde{M}$, (2) $F$ is as in the statement of the theorem, and (3) $\widetilde{j}(\kappa)>F(\lambda)$. Since $\GCH$ implies $\lambda^{<\kappa}=\lambda$ and $F(\lambda)^\lambda=F(\lambda)$, we can further assume without loss of generality that
$$\widetilde{M}=\{\widetilde{j}(h)(\widetilde{j}"\lambda,\alpha)\mid\textrm{$h:P_\kappa\lambda\times\kappa\to \widetilde{V}$ $\and$ $\alpha<F(\lambda)$ $\and$ $h\in \widetilde{V}$}\}.$$
Now let $V=\widetilde{V}[\widetilde{G}]$ be the forcing extension of Lemma \ref{lemmafffeaston}. So $\widetilde{j}$ lifts to $j:V\to M=\widetilde{M}[\widetilde{j}(\widetilde{G})]$ where
$$M=\{j(h)(j"\lambda,\alpha)\mid \textrm{$h:(P_\kappa\lambda)^{\widetilde{V}}\times\kappa$ $\and$ $\alpha<F(\lambda)$ $\and$ $h\in V$}\},$$ 
and there is a partial function $f$ from $\kappa$ to $V_\kappa$ in $V$ such that $j(f)(\kappa)=F$. Working in $V$, define $Y_f\subseteq \dom(f)$ to be the set of all $\gamma<\kappa$ such that the following properties hold.
\begin{enumerate}
\item $\gamma\in\dom(f)$ is an inaccessible cardinal and $f"\gamma\subseteq V_\gamma$.
\item For some regular cardinal $\lambda_\gamma\in [\gamma,\kappa)$, the value of $f(\gamma)$ is a function from $[\gamma,\lambda_\gamma]\cap\REG$ to $\CARD$ satisfying the requirements of Easton's theorem.
\end{enumerate}
Clearly, $Y_f$ must have measure one with respect to the normal measure $U=\{X\subseteq \kappa\mid \kappa\in j(X)\}$. 

Let $\P_{\kappa+1}=\langle (\P_\gamma,\dot{\Q}_\gamma)\mid \gamma<\kappa+1\rangle$ be the length $\kappa+1$ Easton support iteration defined as follows.
\begin{itemize}
\item If $\gamma\in Y_f$ then let $\dot{\Q}_\gamma=\dot{\Q}^{f(\gamma)}_{[\gamma,\lambda_\gamma]}$ be a $\P_\gamma$-name for the Easton-support product $\Q^{f(\gamma)}_{[\gamma,\lambda_\gamma]}$ as defined in $V^{\P_\gamma}$ (see Section \ref{sectionpreliminaries} above for a definition of the poset $\Q^{f(\gamma)}_{[\gamma,\lambda_\gamma]}$). \marginpar{\tiny Alternatively, we could define the iteration using just $f$: if $\P_\gamma$ forces that $\gamma$ is an inaccessible cardinal closed under $f$ and $f(\gamma)$ happens to be a $\P_\gamma$-name for an Easton function on $[\kappa,\lambda_\gamma]$, then let ...}
\item If $\gamma<\kappa$ and $\gamma\notin Y_f$, then $\dot{\Q}_\gamma$ is a $\P_\gamma$-name for trivial forcing.
\item For the stage $\kappa$ forcing, let $\dot{\Q}_\kappa=\dot{\Q}_{[\kappa,\lambda]}^F$ be a $\P_\kappa$-name for the Easton-support product $\Q^{F}_{[\kappa,\lambda]}$ as defined in $V^{\P_\kappa}$.
\end{itemize}
Let $G*H$ be $V$-generic for $\P_\kappa*\dot{\Q}^F_{[\kappa,\lambda]}$. Standard arguments show that $\P_\kappa*\dot\Q^F_{[\kappa,\lambda]}$ preserves cofinalities and that in $V[G][H]$ one has $2^\delta=F(\delta)$ for every regular cardinal $\delta$ in $[\kappa,\lambda]$. Our argument follows Woodin's: we show that there is a further cofinality-preserving forcing extension $V[G][H][g_0]$ in which $\kappa$ remains $\lambda$-supercompact, and the desired continuum function is preserved on the interval $[\kappa,\lambda]$.

\subsection{Lifting the embedding through $\P_\kappa$.}

By elementarity, the forcing $j(\P_\kappa)$ can be factored in $M$ as $j(\P_\kappa)\cong \P_\kappa*\dot{\Q}^{j(f)(\kappa)}_{[\kappa,\lambda_\kappa]}*\dot{\P}_{\tail}$. Since $j(f)(\kappa)=F$ by Lemma \ref{lemmafffeaston}(3), and since $M^\lambda\subseteq M$, it follows that 
$$j(\P_\kappa)\cong\P_\kappa*\dot{\Q}^{F}_{[\kappa,\lambda]}*\dot{\P}_{\tail}.$$
Since the next stage of nontrivial forcing in $j(\P_\kappa)$ beyond $\kappa$ must be beyond $F(\lambda)$, it follows that $\P_{\tail}:=(\dot{\P}_{\tail})_{G*H}$ is ${\leq}F(\lambda)$-closed in $M[G][H]$. 

Now we will show that there is an $M[G][H]$-generic filter for $\P_{\tail}$ in $V[G][H]$. The argument is similar to that in the proof of Lemma \ref{lemmafffeaston}. For $\xi<\kappa$ let $X_\xi$ denote the collection of all nice $\P_\xi$-names for dense subsets of the tail $\dot{\P}_{\xi,\kappa}$ of the iteration $\P_\kappa$. It follows that $X_\xi$ has size at most $2^\kappa$ and hence $X:=\bigcup\{X_\xi\mid\xi<\kappa\}$ has size at most $2^\kappa$. Suppose ${D}\in M[G][H]$ is a dense subset of $\P_{\tail}$ and let $\dot{D}\in M$ be a $\P_\kappa*\dot{\Q}^F_{[\kappa,\lambda]}$-name for $D$. Without loss of generality we may assume that $\dot{D}\in j(X)$ and hence there is a function $h_{\dot{D}}:(P_\kappa\lambda)^{\widetilde{V}}\times\kappa\to X$ in $V$ with $\dot{D}=j(h)(j"\lambda,\alpha)$ for some $\alpha<F(\lambda)$. In other words, every dense subset of $\P_{\tail}$ in $M[G][H]$ has a $\P_\kappa*\dot{\Q}^F_{[\kappa,\lambda]}$-name that is represented by a function $(P_\kappa\lambda)^{\widetilde{V}}\times\kappa\to X$ in $V$. Since there are at most $(2^\kappa)^{\lambda^{<\kappa}}=2^\lambda=\lambda^+$-such functions we can enumerate them as $\<h_\xi\mid\xi<\lambda^+\>\in V$. Working in $V[G][H]$ we build a decreasing sequence $\<p_\xi\mid\xi<\lambda^+\>$ with $p_\xi\in \P_{\tail}$ meeting every dense subset of $\P_{\tail}$ in $M[G][H]$ as follows. Assume $p_\zeta$ has been constructed for $\zeta<\xi$. At stage $\xi<\lambda$, since $M[G][H]^\lambda\subseteq M[G][H]$ in $V[G][H]$, we may let $r_\xi\in \P_{\tail}$ be a lower bound of all previously constructed conditions. Since $\<j(h_\xi)(j"\lambda,\alpha)_{G*H}\mid\alpha<F(\lambda)\>\in M[G][H]$ and $\P_{\tail}$ is ${\leq}F(\lambda)$-closed in $M[G][H]$, it follows that there is a single condition $p_\xi$ below $r_\xi$ meeting every dense subset of $\P_{\tail}$ in the sequence $\<j(h_\xi)(j"\lambda,\alpha)_{G*H}\mid\alpha<F(\lambda)\>$. Since every dense subset of $\P_{\tail}$ in $M[G][H]$ has a name that represented by some function on our list $\<h_\xi\mid\xi<\lambda^+\>$, we can use $\<p_\xi\mid\xi<\lambda^+\>$ to generate an $M[G][H]$-generic filter, call it $G_{\tail}\in V[G][H]$.

Since conditions in $\P_\kappa$ have bounded support, it follows that $j"G\subseteq G*H*G_{\tail}$, and thus we may lift the embedding to $j:V[G]\to M[j(G)]$ in $V[G][H]$ where $j(G)=G*H*G_{\tail}$ (as before, see \cite[Proposition 9.1]{Cummings:Handbook}). We have
$$M[j(G)]=\{j(h)(j"\lambda,\alpha)\mid \textrm{$h:(P_\kappa\lambda)^{\widetilde{V}}\times\kappa\to V[G]$ $\and$ $\alpha<F(\lambda)$ $\and$ $h\in V[G]$}\}.$$

\subsection{Obtaining an $M[j(G)]$-generic filter for $j(\Q^F_{[\kappa,\lambda]})$.}

We will factor the embedding $j:V[G]\to M[j(G)]$ through an ultrapower embedding $j_0$, force over $V[G][H]$ with $j_0(\Q^F_{[\kappa,\lambda]})$ to obtain $g_0$, and then transfer $g_0$ to an $M[j(G)]$-generic filter $g$ for $j(\Q^F_{[\kappa,\lambda]})$. We will also show that forcing with $j_0(\Q^F_{[\kappa,\lambda]})$ over $V[G][H]$ preserves cardinals and does not disturb the continuum function below, or at, $\lambda$.


Let $X=\{j(h)(j"\lambda)\mid h:(P_\kappa\lambda)^{\widetilde{V}}\to V[G] \textrm{ where $h\in V[G]$}\}$. Then it follows that $X\elesub M[j(G)]$. Let $\invMos:M_0'\to M[j(G)]$ be the inverse of the Mostowski collapse $\Mos:X\to M_0'$ and let $j_0:V[G]\to M_0'$ be defined by $j_0:=\pi\circ j$. It follows that $j_0$ is the ultrapower embedding by the measure $U_0:=\{X\subseteq (P_\kappa\lambda)^{\widetilde{V}}\mid j"\lambda\in j(X)\}$ where $U_0\in V[G][H]$. It is shown in \cite{Laver:CertainVeryLargeCardinalsAreNot} that the ground model is definable in any set forcing extension. Thus, by elementarity, $M_0'$ is of the form $M_0[j_0(G)]$, where $M_0\subseteq M_0'$ and $j_0(G)\subseteq j_0(\P_\kappa)\in M_0'$ is $M_0$-generic. Furthermore, $j_0(G)=G*H_0*G^0_{\tail}$ where $H_0$ is $M_0[G]$-generic for $\Mos(\Q^F_{[\kappa,\lambda]})$ and $G^0_{\tail}$ is $M[G][H_0]$-generic for the tail of the iteration $j_0(\P_\kappa)$ above stage $\kappa$. The following diagram is commutative.
$$
\xymatrix{
V[G] \ar[r]^j \ar[rd]_{j_0} & M[j(G)] \\
						 & M_0[j_0(G)] \ar[u]_{\invMos} \\
}
$$
It follows that $j_0$ is a class of $V[G][H_0]$, $M_0[j_0(G)]$ is closed under $\lambda$-sequences in $V[G][H_0]$, and that $j_0(\kappa)>\Mos(F(\lambda))$ with $|j_0(\kappa)|^V=\lambda^+$. Hence the cardinal structures of $M_0[j_0(G)]$ and $V[G][H_0]$ are identical up to and including $\lambda^+$. Furthermore, since $j"\lambda\in X$ it follows that $X^\lambda\subseteq X$ in $V[G][H]$ which implies $\lambda^+\subseteq X$ and $\crit(k)\geq\lambda^+$. Notice that $F\in\dom(\pi)=X$ because $F=j(f)(\kappa)=k(j_0(f))(\kappa)=k(j_0(f)(\kappa))\in\ran(k)=\dom(\pi)$. Thus $\Mos(\Q^F_{[\kappa,\lambda]})=\Q^{\Mos(F)}_{[\kappa,\lambda]}$, where $\pi(F)$ is a function with domain $\pi([\kappa,\lambda])=[\kappa,\lambda]$ satisfying the requirements of Easton's theorem in $M_0[G]$. Let $g_0$ be $V[G][H]$-generic for $j_0(\Q^F_{[\kappa,\lambda]})$.

The next claim shows that forcing with $j_0(\Q^F_{[\kappa,\lambda]})$ over $V[G][H]$ preserves cardinals and does not disturb the continuum function below, or at, $\lambda$.
\begin{claim}\label{claimmild}
$j_0(\Q^F_{[\kappa,\lambda]})$ is ${\leq}\lambda$-distributive and $\lambda^{++}$-c.c. in $V[G][H]$.
\end{claim} 

\begin{proof}[Proof of Claim \ref{claimmild}]
First we will demonstrate the distributivity. By elementarity, $j_0(\Q^F_{[\kappa,\lambda]})$ is ${\leq}\lambda$-closed in $M_0[j_0(G)]$. Since $G^0_{\tail}\in V[G][H_0]$, it follows that $j_0(\Q^F_{[\kappa,\lambda]})\in V[G][H_0]$ and since $M_0[j_0(G)]$ is closed under $\lambda$-sequences in $V[G][H_0]$, it follows that $j_0(\Q^F_{[\kappa,\lambda]})$ is ${\leq}\lambda$-closed in $V[G][H_0]$. Notice that $H_0$ is $V[G]$-generic for $\Mos(\Q^F_{[\kappa,\lambda]})=\Q^{\Mos(F)}_{[\kappa,\lambda]}=\prod_{\gamma\in[\kappa,\lambda]\cap\REG}\Add(\gamma,\pi(F(\gamma)))^{V[G]}$ and that the quotient forcing $\Q^F_{[\kappa,\lambda]}/H_0$ is $\lambda^+$-c.c. in $V[G][H_0]$ since it is isomorphic to $\Q^{F^*}_{[\kappa,\lambda]}$ in that model, where $F^*$ is a function with domain $\REG\cap[\kappa,\lambda]$ such that $F^*(\mu)=0$ if $F(\mu)\leq\lambda^+$ and $F^*(\mu)=F(\mu)$ otherwise. Easton's Lemma states that ${\leq}\lambda$-closed forcing remains ${\leq}\lambda$-distributive after $\lambda^+$-c.c. forcing. Hence $j_0(\Q^F_{[\kappa,\lambda]})$ is ${\leq}\lambda$-distributive in $V[G][H]$. 

Now let us show that $j_0(\Q^F_{[\kappa,\lambda]})$ is $\lambda^{++}$-c.c. in $V[G][H]$. Each $p\in j_0(\Q^F_{[\kappa,\lambda]})$ can be written as $p=j_0(h_p)(j_0"\lambda)$ for some function $h_p:(P_\kappa\lambda)^{\widetilde{V}}\to \Q^F_{[\kappa,\lambda]}$ with $h_p\in V[G]$. For each $p\in j_0(\Q^F_{[\kappa,\lambda]})$ the domain of $h_p$ has size $\lambda$ in $V[G]$, and thus $h_p$ leads to $\overline{h}_p:\lambda\to \Q^F_{[\kappa,\lambda]}$, which can be viewed as a condition in the full support product of $\lambda$ copies of $\Q^F_{[\kappa,\lambda]}$ taken in $V[G]$. Let us denote this product by $\overline{\Q}^F_{[\kappa,\lambda]}$. We will show that $j_0(\Q^F_{[\kappa,\lambda]})$ is $\lambda^{++}$-c.c. in $V[G][H]$ by arguing that $\overline{\Q}^F_{[\kappa,\lambda]}$ is $\lambda^{++}$-c.c. in $V[G][H]$ and that an antichain of $j_0(\Q^F_{[\kappa,\lambda]})$ of size $\lambda^{++}$ in $V[G][H]$ would lead to an antichain of $\overline{\Q}^F_{[\kappa,\lambda]}$ of size $\lambda^{++}$ in $V[G][H]$.

We now briefly describe how to prove that $\overline{\Q}^F_{[\kappa,\lambda]}$ is $\lambda^{++}$-c.c. in $V[G][H]$. An easy delta-system argument shows that $\overline{\Q}^F_{[\kappa,\lambda]}$ is $\lambda^{++}$-c.c. in $V[G]$. Suppose $A$ is an antichain of size $\delta$ of $\overline{\Q}^F_{[\kappa,\lambda]}$ in $V[G][H]$. We will show that $A$ leads to an antichain of $\overline{\Q}^F_{[\kappa,\lambda]}\cong\Q^F_{[\kappa,\lambda]}\times\overline{\Q}^F_{[\kappa,\lambda]}$ in $V[G]$ of size $\delta$. Let 
$$q\forces\textrm{$\dot{A}$ is an antichain of $\textstyle\overline{\Q}^F_{[\kappa,\lambda]}$ and $\dot{f}:\delta\to\dot{A}$ is bijective}$$
where $q\in \Q^F_{[\kappa,\lambda]}$ and $\dot{A}_H=A$. For each $\alpha<\delta$ let $q_\alpha\leq q$ be such that $q_\alpha\forces\dot{f}(\alpha)=\check{p}_\alpha$ where $p_\alpha\in\overline{\Q}^F_{[\kappa,\lambda]}$. We have $\overline{\Q}^F_{[\kappa,\lambda]}\cong\Q^F_{[\kappa,\lambda]}\times\overline{\Q}^F_{[\kappa,\lambda]}$ in $V[G]$, and it is easy to check that $W:=\{(q_\alpha,p_\alpha)\mid\alpha<\delta\}$ is an antichain of $\Q^F_{[\kappa,\lambda]}\times\overline{\Q}^F_{[\kappa,\lambda]}$ in $V[G]$ of size $\delta$. This shows that $\overline{\Q}^F_{[\kappa,\lambda]}$ is $\lambda^{++}$-c.c. in $V[G][H]$.

It remains to show that an antichain of $j_0(\Q^F_{[\kappa,\lambda]})$ in $V[G][H]$ with size $\lambda^{++}$ would lead to an antichain of $\overline{\Q}^F_{[\kappa,\lambda]}$ in $V[G][H]$ of size $\lambda^{++}$. Suppose $A$ is an antichain of $j_0(\Q^F_{[\kappa,\lambda]})$ with size $\delta$ in $V[G][H]$. Each $p\in A$ is of the form $j_0(h_p)(j"\lambda)$ where $h_p:(P_\kappa\lambda)^{\tilde{V}}\to \Q^F_{[\kappa,\lambda]}$. As mentioned above, each $h_p$ leads to a condition $\overline{h}_p\in\overline{\Q}^F_{[\kappa,\lambda]}$. It is easy to check that $\overline{A}:=\{\overline{h}_p\mid p\in A\}$ is an antichain of $\overline{\Q}^F_{[\kappa,\lambda]}$.
\end{proof}

Let us now show that $g_0$ can be transferred along $\invMos$ to an $M[j(G)]$-generic filter for $j(\Q^F_{[\kappa,\lambda]})$.
\begin{claim}\label{claimtransfer}
$\invMos"g_0\subseteq j(\Q^F_{[\kappa,\lambda]})$ generates an $M[j(G)]$-generic filter $g$ for $j(\Q^F_{[\kappa,\lambda]})$.
\end{claim}

\begin{proof}
Suppose $D\in M[j(G)]$ is an open dense subset of $j(\Q^F_{[\kappa,\lambda]})$ and let $D=j(h)(j"\lambda,\alpha)$ for some $h\in V[G]$ with $\dom(h)=(P_\kappa\lambda)^{\widetilde{V}}\times\kappa$ and $\alpha<F(\lambda)$. Without loss of generality, let us assume that every element of the range of $h$ is a dense subset of $\Q^F_{[\kappa,\lambda]}$ in $V[G]$. We have $D=j(h)(j"\lambda,\alpha)=k(j_0(h))(j"\lambda,\alpha)$. Since $\pi(F(\lambda))<j_0(\kappa)$ we may define a function $\widetilde{h}\in M_0[j_0(G)]$ with $\dom(\widetilde{h})=\Mos(F(\lambda))$ by $\widetilde{h}(\xi)=j_0(h)(j_0"\lambda,\xi)$. Then $\dom(k(\widetilde{h}))=k(\Mos(F(\lambda)))=F(\lambda)$ and since the critical point of $k$ is greater than $\lambda$ we have $k(\widetilde{h})(\alpha)=k(j_0(h))(k(j_0"\lambda),\alpha)=j(h)(j"\lambda,\alpha)=D$. Now the range of $\widetilde{h}$ is a collection of $\Mos(F(\lambda))$ open dense subsets of $j_0(\Q^F_{[\kappa,\lambda]})$. Since $j_0(\Q^F_{[\kappa,\lambda]})$ is ${\leq}\Mos(F(\lambda))$-distributive in $M_0[j_0(G)]$, one sees that $\widetilde{D}=\bigcap\ran(\widetilde{h})$ is an open dense subset of $j_0(\Q^F_{[\kappa,\lambda]})$. Hence there is a condition $p\in g_0\cap\widetilde{D}$ and by elementarity, $\invMos(p)\in \invMos"g_0\cap \invMos(\widetilde{D})\subseteq D$.
\end{proof}

\subsection{Performing surgery.}

With the $M[j(G)]$-generic filter $g$ for $j(\Q^F_{[\kappa,\lambda]})$ in hand, we surgically modify $g$ to obtain $g^*$ with $j"H\subseteq g^*$, and then argue that $g^*$ remains an $M[j(G)]$-generic filter for $j(\Q^F_{[\kappa,\lambda]})$.

Now let us define $g^*$. For $p\in g\subseteq j(\Q^F_{[\kappa,\lambda]})$, define a new partial function $p^*$ with $\dom(p^*)=\dom(p)$ by letting $p^*$ be equal to $p$ unless $p$ contradicts $j"H$, in which case we flip the appropriate bits so that $p^*$ is compatible with the elements of $j"H$. More precisely, working in $V[G][H]$, let $Q$ be the partial function with $\dom(Q)\subseteq\dom(j(\Q^F_{[\kappa,\lambda]}))$ defined by $Q=\bigcup j"H$. Given $p\in g$, let $p^*$ be the condition \marginpar{\tiny TYPO. We need to SHOW it is a condition.} in $j(\Q^F_{[\kappa,\lambda]})$ obtained from $p$ by altering $p$ on $\dom(p)\cap\dom(Q)$ so that $p^*$ agrees with $Q$. Let 
$$g^*=\{p^*\mid p\in g\}.$$ Clearly, $j"H\subseteq g^*$, and it remains to argue that $g^*$ remains an $M[j(G)]$-generic filter for $j(\Q^F_{[\kappa,\lambda]})$.

Since conditions $p\in \Q^F_{[\kappa,\lambda]}$ can have size greater than the critical point of $j$, it follows that $j(p)$ need not equal $j"p$. Thus, some of the modifications we made in obtaining $g^*$ occurred off of the range of $j$ on the ghost coordinates, and so Woodin's key lemma does not apply. We will use the following lemma to show that $g^*$ remains an $M[j(G)]$-generic filter.

\begin{lemma}[New Key Lemma]\label{newkeylemmaeaston}
Suppose $B\in M[j(G)]$ with $B\subseteq j(\dom(\Q^F_{[\kappa,\lambda]}))$ and $|B|^{M[j(G)]}\leq j(\lambda)$. Then the set 
$$\mathcal{I}_B=\{\dom(j(q))\cap B\mid \textrm{$q\in \Q^F_{[\kappa,\lambda]}$}\}$$
has size at most $\lambda$ in $V[G][H]$.
\end{lemma}
\begin{proof}
Let $B$ be as in the statement of the lemma and let $B=j(h)(j"\lambda,\alpha)$ where $h:(P_\kappa\lambda)^{\widetilde{V}}\times\kappa\to P_{\lambda^+}(\dom(\Q^F_{[\kappa,\lambda]}))^{V[G]}$, $\alpha<F(\lambda)$, and $h\in V[G]$. Then $\bigcup\ran(h)$ is a subset of $\dom(\Q^F_{[\kappa,\lambda]})$ in $V[G]$ with $|\bigcup\ran(h)|^{V[G]}\leq\lambda$. Since $V[G]\models \lambda^{<\lambda}=\lambda$ by our $\GCH$ assumption, it will suffice to show that 
$$\mathcal{I}_B\subseteq\{j(d)\cap B\mid d\in P_\lambda(\bigcup\ran(h))^{V[G]}\}.$$
Suppose $\dom(j(q))\cap B\in \mathcal{I}_B$ where $q\in \Q^F_{[\kappa,\lambda]}$. We will show that $\dom(j(q))\cap B=j(d)\cap B$ for some $d\in P_\lambda(\bigcup\ran(h))^{V[G]}$. Let $d:=\dom(q)\cap\bigcup\ran(h)$, then $\dom(j(q))\cap B=j(d)\cap B$ since
$$j(d)=\dom(j(q))\cap\bigcup\ran(j(h))\supseteq \dom(j(q))\cap B.$$
\end{proof}

First let us show that the elements of $g^*$ are conditions in $j(\Q^F_{[\kappa,\lambda]})$ by showing that $g^*\subseteq M[j(G)]$. Suppose $p\in g^*$ and let $\mathcal{I}_{\dom(p)}$ be the collection of all possible intersections of $\dom(p)$ with the domains of conditions $j(q)$ where $q\in H$. By the genericity of $H$, it follows that $H$ is a maximal filter on $\Q^F_{[\kappa,\lambda]}$ and therefore,
$$\mathcal{I}_{\dom(p)}=\{\dom(j(q))\cap\dom(p)\mid q\in \Q^F_{[\kappa,\lambda]}\}.$$ 
Since $p$ is a condition in $j(\Q^F_{[\kappa,\lambda]})$ we have $|\dom(p)|^{M[j(G)]}<j(\lambda)$ and hence by Lemma \ref{newkeylemmaeaston}, it follows that $V[G][H]\models |\mathcal{I}_{\dom(p)}|\leq\lambda$. Let $\<I_\alpha\mid\alpha<\lambda\>\in V[G][H]$ be an enumeration of $\mathcal{I}_{\dom(p)}$. For each $\alpha<\lambda$ choose a particular $q_\alpha\in H$ with $\dom(j(q_\alpha))\cap\dom(p)= I_\alpha$. Since $M[j(G)]$ is closed under $\lambda$-sequences in $V[G][H]$, it follows that $\{ j(q_\alpha)\mid\alpha<\lambda\}$ is a directed subset of $j(\Q^F_{[\kappa,\lambda]})$ in $M[j(G)]$. Notice that $j(\Q^F_{[\kappa,\lambda]})$ is ${<}j(\kappa)$-directed closed in $M[j(G)]$, and so we may let $\bar{p}:=\bigcup\{j(q_\alpha)\mid\alpha<\lambda\}$ be the corresponding partial master condition in $j(\Q^F_{[\kappa,\lambda]})$. Since $p,\bar{p}\in M[j(G)]$ one can easily see that $p^*\in M[j(G)]$ since it can be obtained by modifying $p$ to agree with $\bar{p}$.

Now we will show that $g^*$ remains $M[j(G)]$-generic for $j(\Q^F_{[\kappa,\lambda]})$. Suppose $A\in M[j(G)]$ is a maximal antichain of $j(\Q^F_{[\kappa,\lambda]})$. Since $\Q^F_{[\kappa,\lambda]}$ is $\lambda^+$-c.c. in $V[G]$, it follows by elementarity that the set $B:=\bigcup\{\dom(r)\mid r\in A\}$ has size at most $j(\lambda)$ in $M[j(G)]$. Hence it follows from Lemma \ref{newkeylemmaeaston} that the set of intersections
$$\mathcal{I}_B:=\{\dom(j(q))\cap B\mid q\in\Q^F_{[\kappa,\lambda]}\}$$
satisfies $V[G][H]\models|\mathcal{I}_B|\leq\lambda$. Let $\<E_\alpha\mid\alpha<\lambda\>\in V[G][H]$ be an enumeration of $\mathcal{I}_B$. For each $\alpha<\lambda$, choose $q_\alpha\in H$ with $\dom(j(q_\alpha))\cap B=E_\alpha$. Choose $r_\alpha\in g$ with $\dom(r_\alpha)=\dom(j(q_\alpha))$. Define 
$$\Delta_\alpha:=\{(\delta,\xi,\zeta)\in\dom(r_\alpha)\mid r_\alpha(\delta,\xi,\zeta)\neq j(q_\alpha)(\delta,\xi,\zeta)\}.$$
Clearly, $\Delta_\alpha\in M[j(G)]$ since $r_\alpha,j(q_\alpha)\in M[j(G)]$. Then $\<\Delta_\alpha\mid\alpha<\lambda\>\in M[j(G)]$ because $M[j(G)]$ is closed under $\lambda$-sequences in $V[G][H][g_0]$ by the ${\leq}\lambda$-distributivity of the forcing adding $g_0$. Hence
$$\Delta:=\bigcup\{\Delta_\alpha\mid\alpha<\lambda\}$$
is in $M[j(G)]$. The automorphism $\pi_\Delta:j(\Q^F_{[\kappa,\lambda]})\to j(\Q^F_{[\kappa,\lambda]})$ that flips bits of conditions over $\Delta$ is in $M[j(G)]$ since $\Delta\in M[j(G)]$. This implies that $\pi_\Delta^{-1}(A)$ is a maximal antichain of $j(\Q^F_{[\kappa,\lambda]})$ in $M[j(G)]$, and thus, by the $M[j(G)]$-genericity of $g$, it follows that there is a condition $r\in g\cap\pi_\Delta^{-1}(A)$. Furthermore, $r^*=\pi_\Delta(r)\in g^*$ and $r^*$ meets $A$. Hence $g^*$ is $M[j(G)]$-generic.

This shows that $j$ lifts to $j:V[G][H]\to M[j(G)][j(H)]$ where $j(G)=G*H*G_{\tail}$ and $j(H)=g^*$ in $V[G][H][g_0]$.

\subsection{Lifting $j$ through $j_0(\Q^F_{[\kappa,\lambda]})$.}

By Claim \ref{claimmild} above, $j_0(\Q^F_{[\kappa,\lambda]})$ is ${\leq}\lambda$-distributive in $V[G][H]$. By an easy argument involving intersecting $\lambda$ open dense subsets of $j_0(\Q^F_{[\kappa,\lambda]})$ (see \cite[Proposition 15.1]{Cummings:Handbook}), one can show that $j"g_0$ generates an $M[j(G)][j(H)]$-generic filter for $j(j_0(\Q^F_{[\kappa,\lambda]}))$. Thus $j$ lifts to $j:V[G][H][g_0]\to M[j(G)][j(H)][j(g_0)]$ in $V[G][H][g_0]$ witnessing that $\kappa$ is $\lambda$-supercompact in $V[G][H][g_0]$.
\end{proof}

\section{A corollary and an open question}\label{sectionopenquestions}

In this section we will discuss the extent to which our methods for proving Theorem \ref{maintheorem} can be extended to cases in which $\kappa$ is $\lambda$-supercompact where $\lambda$ is a singular cardinal. First, we have an easy Corollary to Theorem \ref{maintheorem}.

\begin{corollary}\label{maincorollary}
Suppose $\GCH$ holds and there is an elementary embedding $j:V\to M$ with critical point $\kappa$ such that for some \emph{singular cardinal} $\lambda>\kappa$ with $\cf(\lambda)<\kappa$ one has $(1)$ $M^\lambda\subseteq M$, $(2)$ $F:[\kappa,\lambda]\cap\REG\to\CARD$ is a function satisfying the requirements of Easton's theorem, and $(3)$ $j(\kappa)>\sup\{F(\gamma)\mid\gamma\in[\kappa,\lambda)\cap\REG\}^+$. Then there is a cofinality-preserving forcing extension in which $\kappa$ remains $\lambda$-supercompact and for every regular cardinal $\gamma\in [\kappa,\lambda]$ one has $2^\gamma=F(\gamma)$.
\end{corollary}

One can see that Corollary \ref{maincorollary} follows directly from Theorem \ref{maintheorem} by extending the function $F$ to $\tilde{F}$ with $\dom(\tilde{F})=[\kappa,\lambda^+]\cap\REG$ as follows. If $F$ is eventually constant on $[\kappa,\lambda)\cap\REG$ then let $\tilde{F}(\lambda^+)$ be this constant value and $\tilde{F}\restrict[\kappa,\lambda)\cap\REG=F$. It follows that the embedding $j$ in the hypothesis of Corollary \ref{maincorollary} witnesses that $\kappa$ is $\lambda^{<\kappa}=\lambda^+$-supercompact, and thus one can apply Theorem \ref{maintheorem} to $j$ and $\tilde{F}$ in order to obtain a forcing extension in which $\kappa$ is $\lambda$-supercompact and $2^\gamma=F(\gamma)$ for all $\gamma\in[\kappa,\lambda)\cap\REG$. In fact, in this case, the hypothesis in Corollary \ref{maincorollary}(3) can be weakened to $j(\kappa)>\sup\{F(\gamma)\mid\gamma\in[\kappa,\lambda)\cap\REG\}$. For the remaining case in which $F$ is not eventually constant, we may extend $F$ to $\tilde{F}$ with $\dom(\tilde{F})=[\kappa,\lambda^+]\cap\REG$ by letting $\tilde{F}(\lambda^+)=\sup\{F(\gamma)\mid \gamma\in[\kappa,\lambda)\cap\REG\}^+$ and then apply Theorem \ref{maintheorem}.

This suggests the following natural question, which our methods do not seem to answer.

\begin{question}\label{mainquestion}
Suppose $\GCH$ holds and $\kappa$ is $\lambda$-supercompact where $\lambda$ is a singular cardinal with $\cf(\lambda)\geq\kappa$ and $F:[\kappa,\lambda)\cap\REG\to\CARD$ is a function such that $\forall\alpha,\beta\in\dom(F)$ $\alpha<\cf(F(\alpha))$ and $\alpha<\beta$ $\implies$ $F(\alpha)\leq F(\beta)$. Suppose the $\lambda$-supercompactness of $\kappa$ is witnessed by $j:V\to M$ and $j(\kappa)>\sup\{F(\gamma)\mid\gamma\in[\kappa,\lambda)\cap\REG\}^+$. Is there a cofinality-preserving forcing extension preserving the $\lambda$-supercompactness of $\kappa$ in which $2^\gamma=F(\gamma)$ for all regular cardinals $\gamma\in[\kappa,\lambda)\cap\REG$?
\end{question}

Under the assumption of Question \ref{mainquestion} we have $\lambda^{<\kappa}=\lambda$. Hence, it seems unlikely that there is an easy argument, similar to the above argument for Corollary \ref{maincorollary}, that would answer Question \ref{mainquestion}. A more promising strategy for answering Question \ref{mainquestion} is to improve Lemma \ref{newkeylemmaeaston} to allow for the surgery argument to be carried out. The main obstacles to carrying out this strategy seem to be that if $\lambda$ is singular then conditions in the Easton-support forcing $\Q^F_{[\kappa,\lambda)}$ can have support cofinal in $\lambda$ and $\Q^F_{[\kappa,\lambda)}$ is merely $\lambda^{++}$-c.c.

We close with another natural question, drawing inspiration from \cite{FriedmanHonzik:EastonsTheoremAndLargeCardinals}, regarding controlling the behavior of the continuum function globally, not just on $[\kappa,\lambda]\cap\REG$, while preserving the $\lambda$-supercompactness of $\kappa$. Notice that in the proof of Theorem \ref{maintheorem}, the forcing iteration up to $\kappa$ is defined in terms of the function $f:\kappa\to V_\kappa$, which was added by forcing such that $j(f)(\kappa)=F$ (see Lemma \ref{lemmafffeaston}). Indeed, in the final model of Theorem \ref{maintheorem}, the behavior of the continuum function below $\kappa$ is dictated by the function $f$, and hence the methods of this paper do not provide an answer to the following.

\begin{question}
Suppose GCH holds, $F:\REG\to\CARD$ is a function satisfying Easton's requirements, and $\kappa$ is $\lambda$-supercompact witnessed by $j:V\to M$ where $j(\kappa)>F(\lambda)$. What additional assumptions will allow one to force the continuum function to agree with $F$ at \emph{every regular cardinal} while preserving the $\lambda$-supercompactness of $\kappa$?
\end{question}

\end{document}